\documentclass[11pt,oneside,a4paper,final]{article}
\usepackage{amsmath,amssymb,makeidx,amsfonts,latexsym,enumerate,amsthm}
\usepackage{cite}

\newtheorem{lemma}{Lemma}[section]
\newtheorem{theorem}{Theorem}[section]
\newtheorem{corollary}{Corollary}[section]
\newtheorem{proposition}{Proposition}[section]

\newtheorem{open problem}{Open problem}[section]
\newtheorem{conjecture}{Conjecture}[section]
\newtheorem*{summary of results}{Summary of results}
\newtheorem{remark}{Remark}[section]
\newtheorem{example}{Example}[section]

\newtheorem{thm1}{Theorem}

\begin{document}

\author{Karl-Olof Lindahl\\
School of Mathematics and Systems Engineering\\
V\"{a}xj\"{o} University, 351 95, V\"{a}xj\"{o}, Sweden\\
\texttt{karl-olof.lindahl@lnu.se}}

\title{Divergence and convergence of conjugacies in non-Archimedean 
dynamics\footnote{\textbf{Published in Contemporary matematics, {A}mer. {M}ath. {S}oc., Vol. 508, pp. 89--109, 2010}}}

\date{September 30, 2008}

\maketitle



%
%

\begin{abstract}
We continue the study in \cite{Lindahl:2004} of the
linearizability near an indifferent fixed point of a power series
$f$, defined over a field of prime characteristic $p$. It is known
since the work of Herman and Yoccoz \cite{HermanYoccoz:1981} in
1981 that Siegel's linearization theorem \cite{Siegel:1942} is
true also for non-Archimedean fields. However, they also showed
that the condition in Siegel's theorem is `usually' not satisfied
over fields of prime characteristic. Indeed, as proven in
\cite{Lindahl:2004}, there exist power series $f$ such that the
associated conjugacy function diverges.
We prove that if the 
degrees of the monomials of a power series $f$ are divisible by
$p$, then $f$ is analytically linearizable.
We find a lower (sometimes the best) bound of the size of the
corresponding linearization disc. In the cases where we find the exact
size of the linearization disc, we show, using the Weierstrass degree of
the conjugacy, that $f$ has an indifferent
periodic point on the boundary. We also give a
class of polynomials containing a 
monomial of degree prime to $p$, such that the conjugacy diverges.
\end{abstract}

\vspace{1.5ex}\noindent {\bf Mathematics Subject Classification
(2000):} 32P05, 32H50, 37F50, 11R58

\vspace{1.5ex}\noindent {\bf Key words:} dynamical system,
linearization, non-Archimedean field

\section{Introduction}

The study of complex dynamical systems of iterated analytic functions
begins with the description of the local behavior near fixed points, see
\cite{Carleson/Gamelin:1991,Beardon:1991,Milnor:2000}. Recall that, given a complete valued field $K$, a power series
$f\in K[[x]]$ of the form
\[
f(z)=\lambda z+a_2z^2+a_3z^3\dots, \quad |\lambda |=1,
\]
is said to be analytically linearizable at the indifferent fixed point  at the origin if there is a convergent power series 
solution $g$ to the following form of the Schr\"{o}der functional equation(SFE)
\begin{equation}\label{schroder functional equation}
g\circ f(x)=\lambda g(x), \quad \lambda =f'(0),
\end{equation}
which conjugates $f$ to its linear part. 
The coefficients of the formal solution $g$ of (\ref{schroder
functional equation}) must satisfy a recurrence relation of the form
\[
b_k=\frac{1}{\lambda (1-\lambda^{k-1})}C_k(b_1,\dots,b_{k-1}).
\]
If $\lambda $ is close to a root of unity, the convergence of $g$ then generates a delicate problem of small divisors. 
 In 1942 Siegel proved in his celebrated paper
\cite{Siegel:1942} that the condition
\begin{equation}\label{Siegel condition}
|1- \lambda^n|\geq Cn^{-\beta} \quad\text{for some real numbers
$C,\beta >0$},
\end{equation}
on $\lambda$ is sufficient for convergence in the complex field
case. Later,  Brjuno \cite{Brjuno:1971} proved that the weaker
condition
\begin{equation}\label{condition brjuno}
-\sum_{k=0}^{\infty} 2^{-k}\log\left (\inf_{1\leq n\leq 2^{k+1}-1}
|1 - \lambda ^{n}| \right) <+\infty,
\end{equation}
is sufficient. 
In fact, for
quadratic polynomials, the Brjuno condition is not only sufficient
but also necessary as shown by Yoccoz \cite{Yoccoz:1988}.

Meanwhile, since the work of Herman and Yoccoz in 1981 \cite{HermanYoccoz:1981},
there has been an increasing interest in the non-Archimedean analogue
of complex dynamics, see e.g.\@ \cite{Arrowsmith/Vivaldi:1993,Arrowsmith/Vivaldi:1994,Benedetto:2001b,Benedetto:2003a,
Bezivin:2004b,Hsia:2000,Khrennikov:2001a,KhrennikovNilsson:2001,
Li:1996a,Li:2002a,Lindahl:2004,Lindahl:2007,Lubin:1994,Rivera-Letelier:2003thesis,Rivera-Letelier:2003,DeSmedtKhrennikov:1997}. Herman and Yoccoz  proved
that Siegel's theorem is true also for non-Archimedean fields.

However, for complete valued fields of prime characteristic,
which are necessarily non-Archimedean, the problem was still open; 
in characteristic $p>0$, the Siegel condition, and even the weaker Brjuno condition, is only satisfied if $\lambda $ is \emph{trivial}, that is, that $|1-\lambda ^n|=1$ $\forall n \geq 1 $.
If $\lambda $ is non-trivial (e.g.\@ in locally compact fields  all $\lambda $ are non-trivial), then $\lambda$
generates a problem of small divisors. One might therefore conjecture,
as Herman \cite{Herman:1986}, that
for a locally compact, complete valued field of prime 
characteristics, the formal conjugacy `usually'  diverges, even for polynomials
of one variable.
Indeed, it was proven in \cite{Lindahl:2004} that in characteristic $p>0$,
like in complex dynamics, the formal solution may diverge also in the one-dimensional case.
On the other hand, in \cite{Lindahl:2004} it was also proven that the
conjugacy may still converge due to considerable cancellation of
small divisor terms. The main theorem of \cite{Lindahl:2004} stated that quadratic polynomials with non-trivial multipliers are linearizable if and only if the characteristic of the ground field char $K=2$. 

In the present paper we present a new class of polynomials that yield divergence.
We also note that  the conjugacy converges for all power series $f\in K[[x]]$, whoose monomials 
are all of degree divisible by char $K=p$. 
Furthermore, in case of 
%
of convergence, we estimate the radius of
convergence for
the corresponding
\emph{semi-disc}, i.e.\@ the maximal disc $V$ such that the semi-conjugacy 
(\ref{schroder functional equation}) holds for all $x\in V$, and the
 \emph{linearization disc}\footnote{Here we use the term `linearization disc' rather than `Siegel disc', because in non-Archimedean dynamics the Siegel disc is 
often refered to as the larger maximal disc on which $f$ is one-to-one.} $\Delta$,
i.e.\@ the maximal disc $U$,
about the origin, such that the full conjugacy
$g\circ f\circ g^{-1}(x)=\lambda x$,
holds for all $x\in U$.
We also give sufficient conditions, related to the Weierstrass  degree of the conjugacy, there
being a periodic point on the boundary of the linearization disc. The first non-Archimedean
results in this direction were obtained by Arrowsmith and Vivaldi \cite{Arrowsmith/Vivaldi:1994} who showed that $p$-adic power functions may have indifferent periodic points on the boundary. 
In fact, we prove the following theorem, see Lemma \ref{lemma weierstrass degree and per points}.
\begin{thm1}
Let $K$ be a complete algebraically closed non-Archimedean field. Let $f\in K[[x]]$ have a linearization disc $\Delta$ about an indifferent fixed point.  
Suppose that $\Delta$ is rational open, and that the radius of the corresponding semi-disc of $f$ is strictly greater than that of $\Delta$, 
then $f$ has an indifferent
periodic point on the boundary of $\Delta$. 
\end{thm1}
This theorem and other results of the present paper, stated below, support the idea that the presence of indifferent periodic points on the boundary of   a linearization disc about an indifferent fixed point is typical in the non-Archimedean setting.

For a more thorough treatment of the problem and its relation to
earlier works on non-Archimedean and complex dynamics the reader
is referred to \cite{Lindahl:2004}. Estimates of $p$-adic linearization discs were obtained in 
\cite{Lindahl:2004cpp}.

\section{Summary of results}
\subsection{Divergence and convergence}

In the present paper we find a new class of polynomials that yield
divergence.
\begin{thm1}
Let char $K=p>0$ and let $\lambda \in K$,
$|\lambda |=1$. Then, polynomials of the form
\[
f(x)=\lambda x + a_{p+1}x^{p+1}\in K[x], \quad
a_{p+1}\neq 0,
\]
are not analytically linearizable at the fixed point at the origin
if $|1-\lambda |<1$.
\end{thm1}

On the other hand we also prove convergence
for all power series $f$ whose monomials are all of degree divisible by
char$K=p$.

\begin{thm1}\label{theorem convergence introduction}
Let char $K=p>0$ and let $\lambda \in K$,
$|\lambda |=1$, but not a root of unity. Then, convergent power
series of the form
\[
f(x)=\lambda x + \sum_{p\mid i}a_ix^i,
\]
are linearizable at the origin.
\end{thm1}

These results indicate that the convergence depends mutually on
the powers of the monomials of $f$ and the characteristic $p$ of
$K$, `good' powers for convergence being those divisible by $p$, `bad' powers being those prime to $p$. However, the blend
of prime and co-prime powers may sometimes yield convergence,
sometimes not, at least for non-polynomial power series as shown
in \cite{Lindahl:2004}.
%
However, there might be a sharp distinction for polynomials:
\begin{open problem}
Let $K$ be of characteristic $p>0$.
Is there a polynomial of the form $f(x)=\lambda x+
O(x^2)\in K[x]$, with $\lambda $ not a root of
unity satisfying $|1-\lambda ^n|<1$ for some $n\geq 1$, and
containing a monomial of degree prime to $p$, such that the formal
conjugacy $g$ converges?
\end{open problem}

\subsection{Estimates of linearization discs and periodic points}\label{sectiom est of Siegel summary}

Let $K$ be a field of prime characteristic $p$. Let
$\lambda \in K$, not a root of unity, be such that
the integer
\begin{equation}\label{definition m}
m=m(\lambda)=\min\{n\in\mathbb{Z}:n\geq 1,|1-\lambda ^n|<1 \},
\end{equation}
exists. The case in which such an $m$ does not exist was treated
in \cite{Lindahl:2004}; it was shown that if $|1-\lambda ^n|=1$ for all $n\geq1$,
then the linearization disc of a power series 
\[
f(x)=\lambda x +a_2x^2+a_3x^3+\dots,
\]
is either the closed or open disc of radius $1/a$ where $a=\sup_{i\geq 2}|a_i|^{1/(i-1)}$.

Note that, by Lemma \ref{lemma distance
char p} below, $m$ is not divisible by $p$. Given $\lambda$ and hence
$m$, the integer $k'$ is defined by
\begin{equation}\label{definition k'}
k'=k'(\lambda)=\min\{k\in\mathbb{Z}:k\geq 1,p|k, m|k-1 \}.
\end{equation}
Let $a>0$ be a real number. We will associate with the pair
$(\lambda,a)$, the family of power series
\begin{equation}
\mathcal{F}_{\lambda,a}^p(K)=\left\{\lambda
x +\sum_{p|i} a_ix^i\in K[[x]]:a=\sup_{i\geq 2}|a_i|^{1/(i-1)}\right\},
\end{equation}
and the real numbers
\begin{equation}\label{definition semi radius}
\rho=\rho(\lambda,a)=\frac{|1-\lambda ^m|^{\frac{1}{mp}}}{a},
\end{equation}
and
\begin{equation}\label{definition radius of the linearization disc}
\sigma=\sigma (\lambda,a)=\frac{|1-\lambda
^m|^{\frac{1}{k'-1}}}{a},
\end{equation}
respectively.

As stated in Theorem \ref{theorem convergence introduction}, power
series in the family $\mathcal{F}_{\lambda,a}^p(K)$ are
linearzable at the origin. Given $f$, the corresponding conjugacy
function $g$, will be defined as the unique power series solution
to the Schr\"{o}der functional equation (\ref{schroder functional
equation}), with $g(0)=0$ and $g'(0)=1$.

In Section \ref{section convergence} we use the ansatz of a power
series solution to the Schr\"{o}der functional equation, to obtain
estimates of the coefficients of $g$, and its radius of
convergence. Moreover, applying a result of Benedetto (Proposition
\ref{proposition one-to-one} below), we estimate the radius of
convergence for the inverse $g^{-1}$. The main result can be
stated in the following way.
\begin{thm1}
Let $f\in \mathcal{F}_{\lambda,a}^p(K)$. Then, the
semi-conjugacy $g\circ f (x)=\lambda g(x)$ holds on the
open disc $D_{\rho}(0)$. Moreover, the full conjugacy $g\circ f
\circ g^{-1}(x)=\lambda x$ holds on $D_{\sigma}(0)$.
\end{thm1}

Under further assumptions on $f$, the linearization disc may contain the
larger disc $D_{\rho}(0)$.
\begin{thm1}
Let $f\in \mathcal{F}_{\lambda,a}^p(K)$ be of the form
\[
f(x)=\lambda x +\sum_{i\geq i_0}a_ix^i,
\]
for some integer $i_0>k'$. Then, the full conjugacy $g\circ f
\circ g^{-1}(x)=\lambda x$ holds on a disc larger than or equal to
$D_{\rho}(0)$ or the closed disc $\overline{D}_{\rho}(0)$,
depending on whether $g$ converges on $\overline{D}_{\rho}(0)$ or
not.
\end{thm1}

Note that the estimate of the linearization disc in Theorem
\ref{theorem convergence introduction} is maximal in the sense
that in $\widehat{K}$, the completion of the algebraic
closure of $K$, quadratic polynomials have a fixed point
on the sphere
$S_{\sigma}(0)$ if $m(\lambda)=1$, breaking the conjugacy there.
In fact, the estimate is maximal in a broader sense, according to
the following theorem.

\begin{thm1}\label{theorem k' as the weierstrass degree summary}
Let $f\in \mathcal{F}_{\lambda,a}^p(K)$. Suppose
$a=|a_{k'}|^{1/(k'-1)}$ and $|a_i|<a^{i-1}$ for all $i<k'$. Then,
$D_{\sigma}(0)$ is the linearization disc of $f$ about the origin. In
$\widehat{K}$ we have  $\deg
(g,{\overline{D}_{\sigma}(0)})=k'$.
Furthermore, $f$ has an indifferent periodic point in
$\widehat{K}$ on the sphere $S_{\sigma}(0)$ of period
$\kappa\leq k'$, with multiplier $\lambda ^{\kappa }$.
\end{thm1}

Here, $\deg(g,D)$ denotes the Weierstrass degree of $g$ on the
disc $D$, as defined in Section \ref{section linearization discs and
periodic points}. The Weierstrass degree is the same as the notion
of degree as 'the number of pre-images of a given point, counting
multiplicity'. Since we assume that $g(0)=0$ and $g'(0)=1$,
$\deg(g,\overline{D}_{\sigma}(0))=k'$ means that in the algebraic
closure $\widehat{K}$, $g$ maps the disc
$\overline{D}_{\sigma}(0)$ onto itself exactly $d$-to-$1$,
counting multiplicity.

The result in Theorem \ref{theorem k' as the weierstrass degree
summary} is based on Lemma \ref{lemma weierstrass degree and per points}
that shows that if there is a shift of the value of the
Weierstrass degree from $1$ to $d>1$,  of the conjugacy function
on a sphere $S$, then there is an indifferent  periodic point of period $\kappa\leq d$,  on the sphere $S$.

\section{Preliminaries}

Throughout this paper $K$ is a field of characteristic
$p>0$, complete with respect to a nontrivial absolute value
$|\cdot |_K$. That is, $|\cdot |_K$ is a multiplicative function
from $K$ to the nonnegative real numbers with $|x|_K=0$
precisely when $x=0$, and nontrivial in the sense that it is not
identically $1$ on $K^*$, the set of all nonzero elements
in $K$. If a field $L$ is equipped with an absolute
value, we say that $L$ is a \emph{valued} field. In fact, all
valued fields of strictly positive characteristic are
non-Archimedean. In what follows, we often use the shorter
notation $|\cdot |$ instead of $|\cdot |_K$.

Recall that a non-Archimedean field is a field $K$ equipped with a
non-trivial absolute value $|\cdot |$, satisfying the following
strong or ultrametric triangle inequality:
\begin{equation}\label{sti}
|x+y| \leq  \max[|x|,|y|],\quad\text{for all $x,y\in K$}.
\end{equation}
One useful consequence of ultrametricity is that for any $x,y\in
K$ with $|x|\neq |y|$, the inequality (\ref{sti}) becomes an
equality. In other words, if $x,y\in K$ with $|x|<|y|$, then
$|x+y|=|y|$.

For a field $K$ with absolute value $|\cdot|$ we define the
\emph{value group} as the image
\begin{equation}\label{def-value group}
|K^{*}|=\{|x|:x\in K^* \}.
\end{equation}
Note that $|K^*|$ is a multiplicative subgroup of the positive
real numbers. We will also consider the full image
$|K|=|K^*|\cup\{0\}$. The absolute value $|\cdot |$ is said to be
\emph{discrete} if the value group is cyclic, that is if there is
is an element $\pi\in K$ such that $|K^{*}|=\{|\pi |^n: n\in
\mathbb{Z}\}$. The absolute value $|\cdot |$ can be extended to an
absolute value on the algebraic closure of $K$. We shall denote by
$\widehat{K}$ the completion of the algebraic closure of $K$ with
respect to $|\cdot|$. The fact that $\widehat{K}$ is algebraically
closed and that $|\cdot |$ is nontrivial forces the value group
$|\widehat{K}^*|_{\widehat{K}}$ to be dense on the positive real
line. In particular, $|\cdot |_{\widehat{K}}$ cannot be discrete.

Standard examples of non-Archimedean fields include the $p$-adics,
see for example
\cite{Koblitz:1984}, and
various function fields, see for example
\cite{Cassels:1986}. The $p$-adics include
the $p$-adic integers and their extensions. These fields are all
of characteristic zero. The most important function fields include
fields of formal Laurent series over various fields. These can be
of any characteristic.

\begin{example}\label{example laurent series}
Let $\mathbb{F}$ be a field of characteristic $p>0$, and fix
$0<\epsilon <1$. Let $K=\mathbb{F}((T))$ be the field of
all formal Laurent series in variable $T$, and with coefficients
in the field $\mathbb{F}$. Then $K$ is also of
characteristic $p$. An element $x\in K$ is of the form
\begin{equation}\label{laurent series}
x=\sum_{j\geq j_0} x_jT^j, \quad x_{j_0}\neq 0, \text{ } x_j\in
\mathbb{F},
\end{equation}
for some integer $j_0\in \mathbb{Z}$. This field is complete with
respect to the absolute value for which
\begin{equation}\label{definition absolute value}
|\sum_{j\geq j_0} x_jT^j|=\epsilon ^{j_0}.
\end{equation}
Note that $j_0$ is the order of the zero (or negative the order of
the pole) of $x$ at $T=0$. Let us also note that $|\cdot |$ is the
trivial valuation on $\mathbb{F}$, the subfield consisting of all
constant power series in $K$. In this case
$|K^*|$ is discrete and consists of all nonzero integer
powers of $\epsilon $. The value group of the completion of the
algebraic closure $|\widehat{K}^*|$ is not discrete and
consists of all nonzero rational powers of $\epsilon$. Moreover,
$K$ can be viewed as the completion of the field of
rational functions over $\mathbb{F}$ with respect to the absolute
value (\ref{definition absolute value}) (see, e.g.\@
\cite{Cassels:1986}).
\end{example}

Given an element $x\in K$ and real number $r>0$ we denote by
$D_{r}(x)$ the open disc of radius $r$ about $x$, by
$\overline{D}_r(x)$ the closed disc, and by $S_{r}(x)$ the sphere
of radius $r$ about $x$. If $r\in|K^*|$ (that is if $r$ is
actually the absolute value of some nonzero element of $K$), we
say that $D_{r}(x)$, $\overline{D}_r(x)$, and $S_r(x)$ are
\emph{rational}. Note that $S_r(x)$ is non-empty if and only if it
is rational. If $r\notin |K^*|$, then we will call
$D_{r}(x)=\overline{D}_r(x)$ an \emph{irrational} disc. In
particular, if $a\in K$ and $r=|a|^s$ for some rational number
$s\in\mathbb{Q}$, then $D_{r}(x)$ and $\overline{D}_r(x)$ are
rational considered as discs in $\widehat{K}$. However, they may
be irrational considered as discs in $K$. Note that all discs are
both open and closed as topological sets, because of
ultrametricity. However, as we will see in Section \ref{section
non-Archimedean power series} below, power series distinguish
between rational open, rational closed, and irrational discs.

Again by ultrametricity, any point of a disc can be considered its
center. In other words, if $b\in D_r(a)$, then $D_r(a)=D_r(b)$;
the analogous statement is also true for closed discs. In
particular, if two discs have nonempty intersection, then they are
concentric, and therefore one must contain the other.

The open and closed unit discs, $D_1(0)$ and $\overline D_1(0)$,
respectively play a fundamental role in non-Archimedean analysis,
because of their algebraic properties. In fact, due to the strong
triangle inequality (\ref{sti}), $\overline D_1(0)$ is a ring and
$D_1(0)$ is the unique maximal ideal in $\overline D_1(0)$. The
corresponding quotient field,
\[
\Bbbk=\overline D_1(0)/D_1(0)
\]
is called the \emph{residue field} of $K$. The residue field
$\Bbbk$ will also be of characteristic $p$. Hence we always have
$\Bbbk\supseteq \mathbb{F}_p$. The absolute value $|\cdot |$ is
trivial on $\Bbbk$. For $x\in\overline D_1(0)$, we will denote by
$\overline x$ the \emph{reduction} of $x$ modulo $D_1(0)$.


\subsection{The formal solution}\label{section the formal solution}

The coefficients of the formal solution of (\ref{schroder
functional equation}) must satisfy the recurrence relation
\begin{equation}\label{bk-equation}
b_k=\frac{1}{\lambda (1-\lambda^{k-1})}\sum_{l=1}^{k-1}b_l%
(\sum\frac{l!}{\alpha_1!\cdot ...\cdot
\alpha_k!}a_1^{\alpha_1}\cdot ...\cdot a_k^{\alpha_k})
\end{equation}
where $\alpha _1,\alpha _2,\dots,\alpha _k$ are nonnegative
integer solutions of
\begin{equation}\label{index-equations}
   \left\{\begin{array}{ll}
            \alpha_1+...+\alpha_k=l,\\
            \alpha_1+2\alpha_2...+k\alpha_k=k,\\
            1\leq l\leq k-1.
        \end{array}
   \right.
\end{equation}

The convergence of $g$ will depend mutually on the denominators
$|1-\lambda ^{k-1}|$ and the factorial terms in
(\ref{bk-equation}). In view of Lemma \ref{lemma distance char p},
the denominator is small if $k-1$ is divisible by $m$ and a large
power of the characteristic $p$. In fact, the conjugacy may
diverge as in Theorem \ref{theorem divergence p+1} below. On the
other hand, the factorial term
\[
\frac{l!}{\alpha_1!\cdot ...\cdot \alpha_k!}
\]
is always an integer and hence of modulus zero or one, depending
on whether it is divisible by $p$ or not. Accordingly, factorial
terms may extinguish small divisor terms as in Theorem
\ref{theorem convergence divisible by p} below.


\subsection{Arithmetic of the multiplier}

Let $\lambda\in S_1(0)$, be an element in the unit sphere. The
geometry of the unit sphere and the roots of unity in $K$
was discussed in \cite{Lindahl:2004}. We are concerned with
calculating the distance
\[
|1-\lambda ^n|, \quad\text{for $n=1,2,\dots$}.
\]
Note that if $x,y\in\overline D_1(0)$, then  $|x-y|<1$ if and only
if the reductions $\overline{x},\overline{y}$ belong to the same
residue class. Consequently,
\begin{equation}\label{distance and residue}
|1-\lambda ^n|<1\quad\iff\quad\overline{\lambda}^n -
1=0\quad\text{in $\Bbbk$}.
\end{equation}
Hence, the behavior of $1-\lambda ^n$ falls into one of two
categories, depending on whether the reduction of $\lambda $ is a
root of unity or not. More precisely we have the following lemma
that was proven in \cite{Lindahl:2004}.
\begin{lemma}[Lemma 3.2 \cite{Lindahl:2004}]\label{lemma distance char p}
Let char $K=p>0$ and let $\Bbbk$ be the residue field of
$K$. Let $\Gamma (\Bbbk)$ be the set of roots of unity in
$\Bbbk$. Suppose $\lambda\in S_1(0)$. Then,
\begin{enumerate}
  \item $\overline{\lambda}\notin\Gamma(\Bbbk)$ $\iff$ $|1-\lambda ^n|=1$
  for all integers $n\geq 1$.
  \item If $\overline{\lambda}\in\Gamma(\Bbbk)$, then the integer $m=\min\{n\in\mathbb{Z}:n\geq 1,|1-\lambda ^n|<1\}$
  exists. Moreover, $p\nmid m$ and
\begin{equation}\label{distance char p}
  \left |1- \lambda ^n \right |=
  \left\{\begin{array}{ll}
         1, & \textrm{ if \quad $m\nmid n$,}\\
        |1-\lambda ^m|^{p^j}, & \textrm{ if \quad $n=map^j$, $p\nmid a$.}
  \end{array}\right.
\end{equation}

\end{enumerate}

\end{lemma}
\vspace{1.5ex} \noindent Note that category 2 in Lemma \ref{lemma
distance char p} is always non-empty since
$\Bbbk\supseteq\mathbb{F}_p$. Moreover, if $\Bbbk\subseteq
\overline{\mathbb{F}}_p$, then $\Gamma(K)=\Bbbk^{*}$ and
all $\lambda\in S_1(0)$ falls into category 2. Consequently
category 1 is empty in this case. This happens for example when
$K$ is locally compact, see e.g.\@ \cite{Cassels:1986}.
\begin{proposition}
Let $K$ be a non-Archimedean field with absolute value $|\cdot |$.
Then $K$ is locally compact (w.r.t. $|\cdot |$) if and only if all
three of the following conditions are satisfied: (i) $K$ is
complete, (ii) $|\cdot |$ is discrete, and (iii) the residue field
is finite.
\end{proposition}
On the other hand if $K$ is algebraically closed, then
$\Bbbk$ is infinite and $K$ cannot be locally compact. In
this case $\Bbbk\supseteq \overline{\mathbb{F}}_p$.

We shall only consider the case in which $\lambda $ belongs to
category 2. The other case was treated in \cite{Lindahl:2004}.

\subsection{Mapping properties}\label{section non-Archimedean power series}

Let $K$ be a complete non-Archimedean field. Let $h$ be a power
series over $K$ of the form
\begin{equation*}
h(x)=\sum_{k=0}^{\infty}c_k(x-\alpha )^k, \quad c_k\in K.
\end{equation*}
Then $h$ converges on the open disc $D_{R_h}(\alpha )$ of radius
\begin{equation}\label{radius of convergence}
R_h = \frac{1}{\limsup |c_k| ^{1/k}},
\end{equation}
and diverges outside the closed disc $\overline{D}_{R_h}(\alpha
)$. The power series $h$ converges on the sphere $S_{R_h}(\alpha
)$ if and only if
\[
\lim_{k\to\infty}|c_k| R_h ^k=0.
\]
The following proposition is useful to estimate the size of a
linearization disc, i.e.\@ the maximal disc on which the full conjugacy,
$g\circ f\circ g^{-1}=\lambda x$,  holds.

\begin{proposition}[Lemma 2.2 \cite{Benedetto:2003a}]\label{proposition-discdegree}
Let $K$ be algebraically closed. Let
$h(x)=\sum_{k=0}^{\infty}c_k(x-x_0)^k$ be a nonzero power series
over $K$ which converges on a rational closed disc
$U=\overline{D}_R(x_0)$, and let $0<r\leq R$. Let
$V=\overline{D}_r(x_0)$ and $V'=D_r(x_0)$. Then
  \begin{eqnarray*}
    s &=& \max\{|c_k|r^k:k\geq 0\},\\
    d &=& \max\{k\geq 0:|c_k|r^k=s\},\quad and\\
    d'&=& \min\{k\geq 0:|c_k|r^k=s\}
  \end{eqnarray*}
are all attained and finite. Furthermore,
\begin{enumerate}[a.]
 \item $s\geq |f'(x_0)|\cdot r$.
 \item if $0\in f(V)$, then $f$ maps $V$ onto $\overline{D}_s(0)$
 exactly $d$-to-1 (counting multiplicity).
 \item if $0\in f(V')$, then $f$ maps $V'$ onto $D_s(0)$
 exactly $d'$-to-1 (counting multiplicity).
\end{enumerate}
\end{proposition}
Benedetto's proof uses the Weierstrass Preparation Theorem
\cite{BoschGuntzerRemmert:1984,FresnelvanderPut:1981,Koblitz:1984}.
We will be interested in the special case in that $c_0=x_0=0$. In
this case, we have the following proposition.


\begin{proposition}\label{proposition one-to-one}
Let $K$ be an algebraically closed complete non-Archimedean field
and let $h(x)=\sum_{k=1}^{\infty}c_kx^k$ be a power series over
$K$.
\begin{enumerate}[1.]

 \item Suppose that $h$ converges on the rational closed disc
  $\overline{D}_R(0)$. Let $0<r\leq R$ and suppose that
  \[
   |c_k|r^k\leq |c_1|r\quad \text{ for all } k\geq 2 .
  \]
 Then $h$ maps the open disc $D_{r}(0)$ one-to-one onto
 $D_{|c_1|r}(0)$. Furthermore, if
 \[
 d = \max\{k\geq 1:|c_k|{r}^k=|c_1| r\},
 \]
 then, $h$ maps the  closed disc $\overline{D}_{r}(0)$ onto
 $\overline{D}_{|c_1|r}(0)$ exactly $d$-to-1 (counting
 multiplicity).

 \item Suppose that $h$ converges on the rational open disc
  $D_R(0)$ (but not necessarily on the sphere $S_R(0)$).
  Let $0<r\leq R$ and suppose that
  \[
   |c_k|r^k \leq |c_1|r\quad \text{ for all } k\geq 2 .
  \]
  Then $h$ maps $D_{r}(0)$ one-to-one
  onto $D_{|c_1|r}(0)$.

\end{enumerate}

\end{proposition}

\begin{lemma}\label{lemma  f one-to-one}
Let $K$ be an algebraically closed complete non-Archimedean field.
Let $f(x)=\lambda x +\sum_{i=2}^{\infty}a_ix^i\in K[[x]]$,
$|\lambda |=1$, be convergent on some non-empty disc. Then, the
real number $a= \sup_{i\geq 2} |a_i|^{1/(i-1)}$ exists and
$|a_i|\leq a^{i-1}$ for all $i\geq 2$. Furthermore, $R_f\geq 1/a$
and $f: D_{1/a}(0)\to D_{1/a}(0)$ is bijective.  If
$|a_i|<a^{i-1}$ for all $i\geq 2$  and $f$ converges on the closed
disc $D_{1/a}(0)$, then $f: \overline{D}_{1/a}(0)\to
\overline{D}_{1/a}(0)$ is bijective. Finally, $f$ cannot be bijective on a (rational)
disc greater than $\overline{D}_{1/a}(0)$.
\end{lemma}
\begin{proof}
As $f$ is convergent, we have $\sup
|a_i|^{1/i}<\infty$ and hence we have that
$\sup |a_i|^{1/(i-1)}< \infty$;
\[
\sup |a_i|^{1/(i-1)} \leq \sup_{|a_i|\leq 1} |a_i|^{1/(i-1)} +
\sup_{|a_i|> 1} \left (|a_i|^{1/i}\right )^{i/(i-1)} \leq 1 +
\left (\sup_{|a_i|> 1} |a_i|^{1/i}\right)^2.
\]
Clearly, $|a_i|\leq \left (\sup_{i\geq 2}
|a_i|^{1/(i-1)}\right)^{i-1}$. Moreover, $R_f=(\limsup
|a_i|^{1/i})^{-1}\geq 1/a$.
That $f: D_{1/a}(0)\to D_{1/a}(0)$ is bijective follows from the
second statement in Proposition \ref{proposition one-to-one}.
\end{proof}

\begin{remark}
Proposition \ref{proposition one-to-one} and Lemma \ref{lemma f
one-to-one} also hold when $K$ is not algebraically closed with
the modification that the mappings $h:D_{r}(0)\to D_{|c_1|r}(0)$
and $f:D_{1/a}(0)\to D_{1/a}(0)$ are one-to-one but not
necessarily surjective; the analogous statement is also true for
the closed discs.
\end{remark}

\begin{remark}\label{remark Lemma bijective}
Note that the disc $D_{1/a}(0)$ in Lemma \ref{lemma  f one-to-one}
may be irrational. Let $K$ be a field such that $|\widehat
{K}^*|=\{\epsilon ^r:r\in\mathbb{Q}\}$ for some $0<\epsilon <1$
(as in Example \ref{example laurent series}).  Let $\beta $ be an
irrational number and let $p_n/q_n$ be the $n$-th convergent of
the continued fraction expansion of $\beta$. Let the sequence
$\{a_i\in K\}_{i\geq 2}$ satisfy
\[
|a_i|= \left \{
\begin{array}{ll}
\left ( 1/\epsilon \right )^{p_n}, & \textrm{if \quad $i-1=q_n$ and $p_n/q_n<\beta $},\\
0, & \textrm{otherwise}.
\end{array}\right.
\]
Then,
\[
\sup |a_i|^{1/(i-1)}=\left ( 1/\epsilon \right )^{\beta}\notin
|\widehat{K}|.
\]
On the other hand, if the sequence $\{a_i\}$ is such that $\max_{i\geq 2}|a_i|^{1/(i-1)}$ exists, then
$a=\max_{i\geq 2}|a_i|^{1/(i-1)}\in |\widehat{K}|$ for any $K$.
This is always the case for polynomials.
\end{remark}

\section{Divergence}

As proven in Section \ref{section convergence}, power series of
the form $f(x)=\lambda x+ O(x^2)$, with monomials of
degree divisible by some nonnegative integer power of $p$ are
analytically linearizable at the origin. In the paper
\cite{Lindahl:2004}, it was proven that quadratic polynomials,
$f(x)=\lambda x+a_2x^2$, with $|1-\lambda|<1$, are not
analytically linearizable at the origin. We will prove another
result in this direction. Polynomials of the form $f(x)=\lambda x
+a_{p+1}x^{p+1}$, with $|1-\lambda|<1$, are not analytically
linearizable in characteristic $p>0$.

The key result is Lemma \ref{lemma absolute value b_jp+1} below.
In that lemma we obtain the exact modulus of a subsequence of
coefficients of the conjugacy function $g$. It turns out that this
subsequence contains sufficiently many small divisor terms to
yield divergence.
\begin{lemma}\label{lemma first nonzero terms of the conjugacy}
Let $K$ be a complete valued field and let $f$  be a power series
of the form
\[
f(x)=\lambda x +\sum_{i\geq i_0}a_ix^{i}\in K[[x]],
\]
where $\lambda \neq 0$ but not a root of unity, $i_0\geq 2$ is an
integer, and $a_{i_0}\neq 0$. Then, the associated formal
conjugacy function is of the form
\[
g(x)=x+\sum_{k\geq i_0}b_kx^k,
\]
where $b_{i_0}=a_{i_0}/\lambda (1-\lambda ^{i_0-1})$.
\end{lemma}
\begin{proof}
By definition, the formal conjugacy $g$ is of the form
$g(x)=x+\sum_{k\geq 2}b_kx^k$. The left hand side of the
Schr\"{o}der functional equation (\ref{schroder functional
equation}) is of the form
\[
g\circ f(x)= \sum_{k=1}^{i_0-1}b_k(\lambda x +
a_{i_0}x^{i_0}+\dots)^k
+O(x^{i_0})=\sum_{k=1}^{i_0-1}b_k(\lambda
x)^k+O(x^{i_0}).
\]
For the right hand side we have
\[
\lambda g(x)=\lambda \sum_{k=1}^{i_0-1}b_k x^k
+O(x^{i_0}).
\]
Recall that $\lambda\neq 0$ is not a root of unity. Identification
term by term yields that $b_k=0$ for $1<k<i_0$. Consequently, for
$k=i_0$, the recursion formula (\ref{bk-equation}) yields
\[
b_{i_0}=b_1a_{i_0}/\lambda (1-\lambda ^{i_0-1}).
\]
But by definition, $b_1=1$. This completes the proof of the lemma.
\end{proof}

\begin{lemma}\label{lemma b_jp+1}
Let char $K=p>0$. Let $f(x)=\lambda x+a_{p+1}x^{p+1}\in
K[x]$, with $|\lambda |=1$ but not a root of unity. Then, the
formal solution $g$ of the SFE
(\ref{schroder functional equation}) has coefficients $b_k$ of the
form
\begin{equation}\label{equation b_jp+1}
b_{jp+1}=\frac{1}{\lambda (1-\lambda ^{jp})}\left [ \sum_{i=\lceil
\frac{j-1}{p+1} \rceil
}^{j-1}b_{ip+1}\binom{ip+1}{j-i}\lambda^{ip+1-(j-i)} a_{p+1}^{j-i}
\right ],
\end{equation}
for all integers $j\geq 1$, $b_1=1$, and $b_k=0$ otherwise.
\end{lemma}
\begin{proof}
First note that the case $k\leq p+1$ follows by
Lemma \ref{lemma first nonzero terms of the conjugacy}.

Now assume that the lemma holds for all $k\leq (n-1)p+1$ for some
$n\geq 2 $. In particular, this means that $b_k=0$ for all
$(i-1)p+1<k<ip+1$ where $1\leq i\leq n-1$. We now consider the
case $(n-1)p+1<k\leq np+1$.
We have regarding the left hand side of the SFE
\[
g\circ f(x)=\sum_{k=1}^{np+1}b_k(\lambda x + a_{p+1}x^{p+1})^k +
O(x^{np+2}),
\]
and by hypothesis,

\begin{equation}\label{equation schroder LHS np+1}
 g\circ f(x)=\sum_{i=0}^{n-1}b_{ip+1}(\lambda x
+ a_{p+1}x^{p+1})^{ip+1}+ \sum_{k=(n-1)p+2}^{np+1}b_k(\lambda
x+a_{p+1})^{k}+ O(x^{np+2}).
\end{equation}
For all integers $i \geq 0$, we have the binomial expansion
\[
(\lambda x + a_{p+1}x^{p+1})^{ip+1}=\sum_{\ell
=0}^{ip+1}\binom{ip+1}{\ell}(\lambda
x)^{ip+1-\ell}(a_{p+1}x^{p+1})^{\ell}.
\]
Consequently, with $\delta=\min\{ip+1,n-i\}$. 
\begin{equation}\label{equation scroder binom np+1 1}
(\lambda x+
a_{p+1}x^{p+1})^{ip+1}=\sum_{\ell=0}^{\delta}\binom{ip+1}{\ell}
\lambda ^{ip+1-\ell} a_{p+1}^{\ell}x^{(i+\ell)p+1}
+O(x^{np+2}).
\end{equation}
Also note that for $k\geq (n-1)p + 2$ we have
\begin{equation}\label{equation scroder binom np+1 2}
(\lambda x + a_{p+1}x^{p+1})^k=(\lambda x)^k +
O(x^{np+2}).
\end{equation}
Combining (\ref{equation schroder LHS np+1}), (\ref{equation
scroder binom np+1 1}), and (\ref{equation scroder binom np+1 2})
we obtain that $g\circ f(x)$ is of the form
\[
\sum_{i=0}^{n-1}b_{ip+1}\sum_{\ell=0}^{\delta}\binom{ip+1}{\ell}
\lambda ^{ip+1-\ell}
a_{p+1}^{\ell}x^{(i+\ell)p+1}+\sum_{k=(n-1)p+2}^{np+1}b_k(\lambda
x)^{k} + O(x^{np+2}).
\]
The right hand side of the SFE is of the form
\[
\lambda g(x)=\lambda\sum_{k=1}^{np+1}b_kx^k +
O(x^{np+2}).
\]
Note that (\ref{equation scroder binom np+1 1}) contains no powers
of $x$ in the closed interval $[(n-1)p+2,np]$. Consequently, for
powers of $x$ in this interval the SFE gives
\[
\sum_{k=(n-1)p+2}^{np}b_k(\lambda
x)^{k}=\lambda\sum_{k=(n-1)p+2}^{np}b_kx^k,
\]
and by identification term by term
\[
b_k=0, \quad\text{for all } (n-1)p + 1< k < np+1.
\]
In the remaining case the $x^{np+1}$-term in (\ref{equation
scroder binom np+1 1}) (if it exists) occur for $\ell=n-i$. Such
an $\ell$ exits if and only if $n-i\leq ip+1$, that is if $i\geq
\lceil \frac{n-1}{p+1}\rceil $. Hence, the equation for the
$x^{np+1}$-term yields
\[
b_{np+1}=\frac{1}{\lambda (1-\lambda ^{np})}\left [ \sum_{i=\lceil
\frac{n-1}{p+1}\rceil
}^{n-1}b_{ip+1}\binom{ip+1}{n-i}\lambda^{ip+1-(n-i)} a_{p+1}^{n-i}
\right ],
\]
as required.
\end{proof}

\begin{lemma}\label{lemma-summands b_{jp+1}}
For each summand in the recursion formula (\ref{equation b_jp+1}),
the total power of $a_{p+1}$ is $j$, i.e each $b_{jp+1}$ is of the
form
\begin{equation}\label{power of a_{p+1}}
b_{jp+1}=c_{jp+1}a_{p+1}^{j},
\end{equation}
where where $c_{jp+1}$ is a sum of products with factors of the
form
\begin{equation}\label{factors}
\frac{1}{\lambda(1-\lambda^{\alpha})} \binom{a}{b}\lambda^{\beta}.
\end{equation}
\end{lemma}
\begin{proof}
First note that in view of (\ref{equation
 b_jp+1}) we have for $j=1$ that the coefficient
$b_{p+1}=a_{p+1}/(\lambda(1-\lambda ^p))$. Now assume that
(\ref{power of a_{p+1}}) holds for $j\leq n$. For $j=n+1$ and
$1\leq i\leq n$, we then have regarding the right hand side of
(\ref{equation b_jp+1}) that
\begin{equation*}
 b_{ip+1}a_{p+1}^{n+1-i}=c_{ip+1}a_{p+1}^{i}a_{p+1}^{n+1-i}=c_{ip+1}a_{p+1}^{n+1},
\end{equation*}
as required.
\end{proof}

\begin{lemma}\label{lemma absolute value b_jp+1}
Let $f(x)=\lambda x+a_{p+1}x^{p+1}$, $|\lambda |=1$ but not a root
of unity. Suppose $|1-\lambda |<1$. Then the formal solution $g$
of the SFE (\ref{schroder functional equation}) has coefficients
$b_k$ of absolute value
\begin{equation}\label{equation absolute value b_jp+1}
|b_{jp+1}|=\frac{|a_{p+1}|^j}{\prod_{i=1}^{j}|1-\lambda ^{ip}|},
\end{equation}
for all integers $j\geq 1$, $b_1=1$, and $b_k=0$ otherwise.

\end{lemma}
\begin{proof}
By Lemma \ref{lemma-summands b_{jp+1}}, the total
power of $a_{p+1}$ in the products $b_{ip+1}a_{p+1}^{j-i}$ of the
right hand side of (\ref{equation b_jp+1}) is $j$. Also recall
that $|\lambda |=1$. Furthermore
\[
\left |\binom{ip+1}{1} \right |=1,
\]
for all nonnegative integers $i$. Now, since $|1-\lambda |<1$, it
follows by induction over $j$ that the $b_{(j-1)p+1}$-term in
(\ref{equation b_jp+1}) is strictly greater than all the others
and by ultrametricity  we obtain
\[
|b_{jp+1}|=\left |\frac{1}{\lambda (1-\lambda
^{jp})}b_{(j-1)p+1}\binom{(j-1)p+1}{1}\lambda^{(j-1)p}
a_{p+1}\right | =\frac{|a_{p+1}|^j}{\prod_{i=1}^{j}|1-\lambda
^{ip}|}.
\]
This completes the proof of Lemma \ref{lemma absolute value
b_jp+1}.
\end{proof}

\begin{lemma}\label{lemma product small divisors}
Let char $K=p>0$ and $\lambda\in K$,
$|\lambda|=1$, but not a root of unity. Suppose $|1-\lambda |<1$,
then
\begin{equation}\label{product small divisors}
\prod_{i=1}^{p^{N-1}}|1-\lambda ^{ip}|=|1-\lambda
|^{p^N(\frac{p-1}{p}(N-1)+1)},
\end{equation}
for all integers $N\geq 1$.
\end{lemma}
\begin{proof}
First note that for each $n=1,\dots,N-1$,  the
number of elements in $\{1,2,\dots,p^N\}$ that are divisible by
$p^n$ but not by $p^{n+1}$, are given by the number $p^N/p^n-p^N/p^{n+1}$.
Since $|1-\lambda |<1$ we can apply Lemma \ref{lemma distance char
p} in the case $m=1$ to obtain
\[
\prod_{i=1}^{p^{N-1}}|1-\lambda ^{ip}|=|1-\lambda
|^{p^N+\sum_{n=1}^{N-1}(\frac{p^N}{p^n}-\frac{p^N}{p^{n+1}})p^n}=|1-\lambda
|^{p^N(1+\frac{p-1}{p}(N-1))},
\]
as required.
\end{proof}

\begin{theorem}\label{theorem divergence p+1}
Let char $K=p>0$ and $f(x)=\lambda x + a_{p+1}x^{p+1}\in
K[x]$. Suppose $|\lambda |=1$ and $|1-\lambda |<1$. Then
$f$ is not analytically linearizable at $x=0$.
\end{theorem}
\begin{proof}
Let $j=p^{N-1}$ for some integer $N\geq 1$. By Lemma \ref{lemma
absolute value b_jp+1},
\[
|b_{p^N+1}|=\frac{|a_{p+1}|^{p^{N-1}}}{\prod_{i=1}^{p^{N-1}}|1-\lambda
^{ip} |}.
\]
But in view of Lemma \ref{lemma product small divisors} this means
that
\[
|b_{p^N+1}|=\frac{|a_{p+1}|^{p^{N-1}}}{|1-\lambda
|^{p^N(1+\frac{p-1}{p}(N-1))}},
\]
and since $|1-\lambda |<1$,
\[
\lim_{N\to\infty} |b_{p^N+1}|^{1/(p^N+1)}=\infty.
\]
Consequently, $\limsup |b_k|^{1/k}=\infty$ so that the conjugacy
diverges.
\end{proof}

\vspace{1.5ex}\noindent Further investigation has led us to
believe that we have divergence also in the case $m>1$ in Theorem
\ref{theorem divergence p+1}, that is $m>1$ is the smallest
integer such that $|1-\lambda ^m|<1$. However, the proof seems to
be more complicated; we do not necessarily have a strictly
dominating term in the right hand side of (\ref{equation b_jp+1})
for all $j$, and consequently, (\ref{equation absolute value
b_jp+1}) may no longer be valid.

\begin{conjecture}[Generalization of Theorem \ref{theorem divergence
p+1}]\label{conjecture divergence p+1 m>1} Let char
$K=p>0$ and $f(x)=\lambda x + a_{p+1}x^{p+1}\in
K[x]$. Suppose $|\lambda |=1$ and that $m\geq 1$ is the
smallest integer such that $|1-\lambda ^m |<1$. Then $f$ is not
analytically linearizable at $x=0$.

\end{conjecture}

\section{Estimates of linearization discs}\label{section convergence}




In this section we consider power series in the family
\begin{equation}
\mathcal{F}_{\lambda,a}^p(K)=\left\{\lambda
x +\sum_{p|i} a_ix^i\in K[[x]]:a=\sup_{i\geq 2}|a_i|^{1/(i-1)}\right\},
\end{equation}
as defined in Section \ref{sectiom est of Siegel summary}. Note
that each $f\in \mathcal{F}_{\lambda,a}^p(K)$ is
convergent on $D_{1/a}(0)$ and by Lemma \ref{lemma  f one-to-one}
$f:D_{1/a}(0)\to D_{1/a}(0)$ is one-to-one. We will prove in
Theorem \ref{theorem convergence divisible by p} that each each
$f\in \mathcal{F}_{\lambda,a}^p(K)$ is linearizable at
the fixed point at the origin. We also estimate the region of
convergence for the corresponding conjugacy function $g$, and its
inverse.

\begin{remark}
By the conjugacy relation $g\circ f\circ g^{-1}(x)=\lambda x $, $f$ must certainly be one-to-one 
on the linearization disc. Consequently, by Lemma \ref{lemma  f one-to-one}, the full conjugacy relation cannot hold on a disc greater than $\overline{D}_{1/a}(0)$. 
\end{remark}

We begin by proving the following, simple but important fact.
Given a power series $f\in \mathcal{F}_{\lambda,a}^p(K)$, the conjugacy
function $g$ only contains monomials of degree divisible by some
nonnegative integer power of $p$. More precisely, we have the
following Lemma.

\begin{lemma}\label{lemma form of conjugacy p|i}
Let char $K=p>0$ and let $f\in
\mathcal{F}_{\lambda,a}^p(K)$, where $|\lambda| =1$ but
not a root of unity. Then, the formal
conjugacy $g$ 
is of the form
\begin{equation}\label{conjugacy with powers divisible by p}
g(x)=x+\sum_{p\mid k}b_kx^k.
\end{equation}
\end{lemma}
\begin{proof}
Let $f\in \mathcal{F}_{\lambda,a}^p(K)$ and let
$f_p(x)=\sum_{(i,p)>1}a_{i}x^{i}$. Let the conjugacy
$g(x)=x+\sum_{k=2}^{\infty}b_kx^k$. We will prove by induction
that if $g$ is the formal solution of the Schr\"{o}der functional
equation, then $b_k=0$ for all $k\geq 2$ such that $p\nmid k$.

By definition $b_1=1$ and by Lemma \ref{lemma first nonzero terms
of the conjugacy} $b_k=0$ for $1<k<p$. Assume that $b_k=0$ for all
$k>1$ such that $(j-1)p< k < jp$, where $1\leq j\leq N$. Let $h_N$
be the polynomial
\[
h_N(x)=\lambda x + f_p(x) + \sum_{i=1}^{N}b_{ip}(\lambda x +
f_p(x))^{ip} \mod x^{(N+1)p}.
\]
Note that the terms
\[
(\lambda x + f_p(x))^{ip}=\sum_{l=0}^{ip}\binom{ip}{l}(\lambda
x)^{l}(f_p(x))^{ip-l},
\]
do only contain powers of $x$ divisible by $p$. This follows
because, in characteristic $p$, $\binom{ip}{l}=0$ if $p\nmid l$.
Accordingly, $h_N(x)-\lambda x$ contains only powers of $x$
divisible by $p$. By hypothesis,
\[
g\circ f(x)=h_N(x) + \sum_{k=Np+1}^{(N+1)p-1}b_k(\lambda x +
f_p(x))^k + O(x^{(N+1)p}).
\]
As noted above the monomials of $f_p$ are of degree greater than
or equal to $p$. Consequently,
\[
g\circ f(x)=h_N(x) + \sum_{k=Np+1}^{(N+1)p-1}b_k(\lambda x)^k +
O(x^{(N+1)p}).
\]
Recall that, $h_N(x)-\lambda x$ contains only powers of $x$
divisible by $p$. Consequently, identification term by term with
the right hand side $\lambda g(x)$ yields that $b_k=0$ for all
$Np<k<(N+1)p$ as required.
\end{proof}

In the following we shall estimate the coefficients of the
conjugacy (\ref{conjugacy with powers divisible by p}). As noted
in Section \ref{section the formal solution} these coefficients
are given by the recursion formula (\ref{bk-equation}). Our main
result (Theorem \ref{theorem convergence divisible by p}) of this
section is based on the estimate obtained in Lemma \ref{lemma b_k
estimate p|k} below.

In preparation, we recall the following definitions from Section
\ref{sectiom est of Siegel summary}. The integer $m$ is defined by
\begin{equation}\label{definition m}
m=m(\lambda)=\min\{n\in\mathbb{Z}:n\geq 1,|1-\lambda ^n|<1 \}.
\end{equation}
Note that, by Lemma \ref{lemma distance char p}, $m$ is not
divisible by $p$. Given $m$, the integer $k'$ is defined by
\begin{equation}\label{definition k'}
k'=k'(\lambda)=\min\{k\in\mathbb{Z}:k\geq 1,p|k, m|k-1 \}.
\end{equation}
Note the following lemma.
\begin{lemma}\label{lemma p|l and m|l-1}
Let $k\geq 2$ be an integer. Then,
\[
\lfloor (k-k')/mp +1  \rfloor
\]
is the the number of positive integers $l \leq k$ that satisfies
the two conditions $p\mid l$ and $m\mid l-1$.
\end{lemma}
\begin{proof}
Let $k'$ be the smallest positive integer such that $p\mid k'$ and
$m\mid k'-1$. Let $\mathbb{Z}_m$ be the residue class modulo $m$.
Also recall that by definition $p\nmid m$. Accordingly, $k'=j'p$
where $j'$ is the unique solution in $\mathbb{Z}_m$ to the
congruence equation
\begin{equation}\label{equation congruence (k,p)>1}
xp-1\equiv 0 \mod m.
\end{equation}
The integer solutions of (\ref{equation congruence (k,p)>1}) are
thus of the form
\[
x=j'+mn,
\]
where $n$ runs over all the integers. It follows that an integer
$l$ satisfies the two conditions $p\mid l$ and $m\mid l-1$ if and
only if it is of the form
\[
l=(j'+ mn)p=k'+ mpn,
\]
for some integer $n$. Given $k\geq 2$, let $t \geq 0$ be the
largest integer such that
\[
k' + (t-1)mp\leq k.
\]
It follows that
\[
t=\lfloor (k-k')/mp +1  \rfloor,
\]
as required.
\end{proof}

\begin{lemma}\label{lemma b_k estimate p|k}
Let $f\in \mathcal{F}_{\lambda,a}^p(K)$.
Then, the
coefficients of the corresponding conjugacy function satisfies
\begin{equation}\label{estimate b_k (k,p)>1}
|b_k|\leq \frac{a^{k-1}}{|1-\lambda ^m|^{\lfloor (k-k')/mp +1
\rfloor}}.
\end{equation}
for all $k \geq 2$.
\end{lemma}
\begin{proof}
Recall that the coefficients $|b_k|$ are given by the recursion
formula (\ref{bk-equation}) where each factorial term
$l!/\alpha_1!\cdot ...\cdot \alpha_k!$ is an integer and thus of
modulus zero or one, depending on whether it is divisible by $p$ or
not. By Lemma \ref{lemma form of conjugacy p|i}, $b_k=0$ if $k\geq
2$ and $p\nmid k$. If $p\mid k$, we have in view of Lemma
\ref{lemma distance char p} that
\begin{equation}\label{distance 1 - lambda k-1 char p}
  |1- \lambda ^{k-1} |=
  \left\{\begin{array}{ll}
         1, & \textrm{ if \quad $m\nmid k-1$,}\\
        |1-\lambda ^m|, & \textrm{ if \quad $m\mid k-1$}.
  \end{array}\right.
\end{equation}
Also recall that $|a_i|\leq a^{i-1}$.  It follows by Lemma
\ref{lemma p|l and m|l-1} that
\[
|b_k|\leq |1-\lambda ^m|^{-\lfloor (k-k')/mp +1 \rfloor}
a^{\alpha},
\]
for some integer $\alpha$. In view of equation
(\ref{index-equations}) we have
\[
\sum_{i=2}^{k}(i-1)\alpha_i=k-l.
\]
Consequently, since $|a_i|\leq a^{i-1}$, we obtain
\begin{equation}\label{estimate prod a_i}
\prod_{i=2}^k|a_i|^{\alpha_i}\leq
\prod_{i=2}^{k}a^{(i-1)\alpha_i}=a^{k-l}.
\end{equation}
Now we use induction over $k$. By definition $b_1=1$ and,
according to the recursion formula (\ref{bk-equation}), $|b_2|\leq
|1-\lambda^m |^{-\lfloor (2-k')/mp +1 \rfloor }|a|$. Suppose that
\[
|b_l|\leq |1-\lambda ^m|^{-\lfloor (l-k')/mp +1 \rfloor }a^{l-1}
\]
for all $l<k$. Then
\[
|b_k|\leq |1-\lambda ^m|^{-\lfloor (k-k')/mp +1 \rfloor
}a^{l-1}\max\left\{ \prod_{i=2}^k|a_i|^{\alpha_i}\right \},
\]
and the lemma follows by the estimate (\ref{estimate prod a_i}).
\end{proof}

\vspace{1.5ex}\noindent The above estimate of $|b_k|$ is maximal
in the sense that we may have equality in (\ref{estimate b_k
(k,p)>1}) for $k=k'$ as shown by the following example.

\begin{example}\label{example a k'x^k'}
Let $f$ be
of the form
\[
f(x)=\lambda x + a_{k'}x^{k'}.
\]
Then, $a=|a_{k'}|^{1/(k'-1)}$. Also note that by Lemma \ref{lemma
first nonzero terms of the conjugacy}
\[
b_{k'}=a_{k'}/\lambda (1-\lambda ^{k'-1}),
\]
and consequently,
\[
|b_{k'}|=a^{k'-1}/|1-\lambda ^m|.
\]
Thus we have equality in (\ref{estimate b_k (k,p)>1}) for $k=k'$
in this case.
\end{example}

By the estimates in Lemma \ref{lemma b_k estimate p|k} and
application of Proposition \ref{proposition one-to-one}, the
radius of
convergence of the conjugacy function $g$ and its inverse $g^{-1}$
can now be estimated by
\begin{equation}\label{definition semi radius}
\rho=\frac{|1-\lambda ^m|^{\frac{1}{mp}}}{a},
\end{equation}
and
\begin{equation}\label{definition radius of the linearization disc}
\sigma=\frac{|1-\lambda ^m|^{\frac{1}{k'-1}}}{a},
\end{equation}
respectively. In other words, we have the following theorem.

\begin{theorem}\label{theorem convergence divisible by p}
Let $f\in \mathcal{F}_{\lambda,a}^p(K)$.
Then, $f$ is analytically linearizable at $x=0$.
The semi-conjugacy relation
  $g\circ f (x)=\lambda g(x)$ holds on $D_{\rho}(0)$. Moreover,
the full conjugacy $g\circ f \circ g^{-1}(x)=\lambda x$ holds on
$D_{\sigma}(0)$. The latter estimate is maximal in the sense that there
exist examples of such $f$ which have a periodic point on the
sphere $S_{\sigma}(0)$.

\end{theorem}

\begin{remark}
Note that $g:D_{\sigma}(0)\to D_{\sigma}(0)$ is bijective if we
consider $D_{\sigma}(0)$ in the algebraic closure
$\widehat{K}$.
\end{remark}

\begin{proof} In view of Lemma \ref{lemma b_k estimate
p|k}, $g$ converges on the open disc of radius
\[
1/\limsup|b_k|^{1/k}\geq a^{-1} |1-\lambda ^m|^{1/mp}=\rho .
\]
Given $m\geq 1$, by definition $k'$ must be one of the numbers
\[
p,\quad 2p, \quad \dots,\quad mp,
\]
and accordingly,
\begin{equation}\label{inequality k'-1<mp}
k'-1<mp.
\end{equation}
To estimate the radius of the maximal disc on which $g$ is
one-to-one we consider
\[
 \sigma_0 = a^{-1} \inf_{k\geq
2}|1-\lambda ^m|^{\lfloor (k-k')/mp +1 \rfloor /(k-1)} \leq
\inf_{k\geq 2}\left (\frac{|b_1|}{|b_{k}|}\right)^{1/(k-1)}.
\]
We will show that the maximum value of the exponent
\[
\delta(k)=\lfloor (k-k')/mp +1 \rfloor /(k-1)
\]
is attained if and only if $k=k'$. First note that $\delta(k)=0$
for all $2\leq k<k'$. In view of (\ref{inequality k'-1<mp}), we have for each integer $n\geq 1$
that
\[
\delta(k'+nmp)-\frac{1}{k'-1}=\frac{n((k'-1)-mp)}{(k'-1+nmp)(k'-1)}<0.
\]
Moreover,
\[
\delta (k)< \delta (k'+ nmp), \quad\textrm{if $mnp< k-k'<
(n+1)mp$}.
\]
Hence, the maximum of $\delta $ is attained if and only if $k=k'$
so that
\[
\sigma_0 =a^{-1} |1-\lambda ^m|^{\frac{1}{k'-1}}=\sigma.
\]
Note that in view of (\ref{inequality k'-1<mp}) $\sigma < \rho $
so that $g$ certainly converges on the closed disc
$\overline{D}_{\sigma}(0)$. Also note that, in terms of $\delta$,
we have by (\ref{estimate b_k (k,p)>1}) that
\[
|b_k|\leq a^{k-1} |1-\lambda ^m|^{-(k-1)\delta (k) }.
\]
Consequently, according to the derived properties of $\delta $,
\begin{equation}\label{estimate sigma (k,p)>1}
|b_k|\sigma^k\leq a^{-1} |1-\lambda
^m|^{\frac{1}{k'-1}+(k-1)(\frac{1}{k'-1}- \delta (k)) }
 \leq
\sigma = |b_1|\sigma.
\end{equation}
In view of Proposition \ref{proposition one-to-one} $g:
D_{\sigma}(0)\to D_{\sigma}(0)$ is a bijection if we consider
$D_{\sigma}(0)$ in $\widehat{K}$. Consequently, $g:
D_{\sigma}(0)\to D_{\sigma}(0)$ is one-to-one if we consider
$D_{\sigma}(0)$ in $K$.

Recall that by Lemma \ref{lemma  f one-to-one} $f:D_{1/a}(0)\to
D_{1/a}(0)$ is one-to-one.
Moreover, $1/a>\rho
>\sigma$. It follows that the
semi-conjugacy and the full conjugacy holds on $D_{\rho}(0)$ and
$D_{\sigma}(0)$ respectively.

That this estimate of $\sigma$ is maximal follows from the
following example. Let char$K=2$ and $f(x)=\lambda x+
a_2x^2\in K[x]$, where $|1-\lambda |<1$. Then $m=1$ and
$k'=p=2$ so that $\sigma = |1-\lambda |/|a_2|$. But
$\hat{x}=(1-\lambda)/a_2$ is a fixed point of $f$, breaking the
conjugacy on $S_{\sigma }(0)$. This completes the proof of Theorem
\ref{theorem convergence divisible by p}. \end{proof}

\vspace{1.5ex}\noindent If $b_{k'}=0$, then we can extend the
estimate of the full conjugacy to the disc $D_{\rho}(0)$.

\begin{lemma}\label{lemma bk'=0}
Let $f\in \mathcal{F}_{\lambda,a}^p(K)$. Suppose the
coefficient $b_{k'}$ of the conjugacy function $g$ is equal to
zero.
Then, the full conjugacy $g\circ f \circ g^{-1}(x)=\lambda x$
holds on a disc larger than or equal to $D_{\rho}(0)$ or
$\overline{D}_{\rho}(0)$, depending on whether $g$ converges on the
closed disc $\overline{D}_{\rho}(0)$ or not.
\end{lemma}
\begin{proof} Recall that $k'$ is the positive
integer defined by (\ref{definition k'}). Assume that $b_{k'}=0$.
Let $k''>k'$ be the integer
\[
k'':=\min\{k\in \mathbb{Z}:k>k',b_{k}\neq 0,p\mid k,m\mid k-1 \}.
\]
In the same way as in Lemma \ref{lemma b_k estimate p|k},
\[
|b_k|\leq \frac{a^{k-1}}{|1-\lambda ^m|^{(k-1)\gamma(k)}, }
\]
where
\[
\gamma (k)= \left \{
\begin{array}{ll} \lfloor (k-k'')/mp +1
\rfloor/(k-1), & \quad k\geq k'',\\
0, & \quad k<k''.
\end{array}\right.
\]
It follows from the proof of Lemma \ref{lemma p|l and m|l-1} that
$ k''=k'+mpn$,  for $n= 1$.  Hence,
\[
k''\geq p + mp.
\]
Consequently,
\[
\gamma (k)-\frac{1}{mp}\leq
\frac{k-k''}{(k-1)mp}+\frac{1}{k-1}-\frac{1}{mp}=\frac{mp-(k''-1)}{(k-1)mp}<0,
\]
for $k\geq 2$. Moreover,
\[
\lim_{k\to\infty}\gamma (k)=1/mp
\]
so that $\sup \gamma (k)=1/mp$. Accordingly, $\max \gamma (k)$
does not exist. Also note that for $\rho =a^{-1}|1-
\lambda ^m|^{1/mp}$, we have
\[
|b_k|\rho ^k\leq a^{-1}|1-\lambda
^m|^{\frac{1}{mp}+(k-1)(\frac{1}{mp}-\gamma (k))}.
\]
Hence, according to the derived properties of $\gamma $,
\begin{equation}\label{inequality g one to one  k''}
|b_k|\rho ^k <\rho =|b_1|\rho,
\end{equation}
for $k\geq 2$. It follows that $g:D_{\rho}(0)\to D_{\rho}(0)$ is
one-to-one. If $g$ converges on the closed disc
$\overline{D}_{\rho}(0)$, then, since we have strict inequality in
(\ref{inequality g one to one  k''}), we also have that  $g:\overline{D}_{\rho}(0)\to
\overline{D}_{\rho}(0)$ is one-to-one.

Recall that $f:D_{1/a}(0)\to D_{1/a}(0) $ is one-to-one and that
$1/a>\rho$. Consequently, the full conjugacy $g\circ f \circ
g^{-1}(x)=\lambda x$ holds on a disc larger than or equal to
$D_{\rho}(0)$ or $\overline{D}_{\rho}(0)$, depending on whether $g$
converges the closed disc $\overline{D}_{\rho}(0)$ or not.
\end{proof}

\vspace{1.5ex}\noindent Note that a sufficient condition that
$b_{k'}=0$ is that $a_i=0$ for all $2\leq i\leq k'$. In fact, in
view of Lemma \ref{lemma first nonzero terms of the conjugacy} we
and the previous lemma we have the following result.

\begin{theorem}
Let $f\in \mathcal{F}_{\lambda,a}^p(K)$ be of the form
\[
f(x)=\lambda x +\sum_{i\geq i_0}a_ix^i,
\]
for some integer $i_0>k'$. Then,
the full conjugacy $g\circ f \circ g^{-1}(x)=\lambda x$ holds on a
disc larger than or equal to $D_{\rho}(0)$ or
$\overline{D}_{\rho}(0)$, depending on whether $g$ converges on the
closed disc $\overline{D}_{\rho}(0)$ or not.
\end{theorem}




\section{Linearization discs and periodic points}\label{section linearization discs and periodic points}

In this section we give sufficient conditions under which the
estimate $\sigma$ in Theorem \ref{theorem convergence divisible by
p}
is maximal in the sense that $D_{\sigma}(0)$ is equal to the
linearization disc. In other words, we give sufficient conditions under
which $D_{\sigma}(0)$ is  the maximal disc $U\subset K$
about the fixed point $x=0$, such that the conjugacy $g\circ
f\circ g^{-1}(x)=\lambda x$ holds for all $x\in U$.

In theorem \ref{theorem convergence divisible by p} we proved that
the estimate $\sigma$ is maximal in the sense that, in the the
special case that $f$ is quadratic and $|1-\lambda|<1$ ($m=1$),
there is a fixed point in the algebraic closure
$\widehat{K}$ on the sphere $S_{\sigma}(0)$, breaking the
conjugacy there. Using the notion of Weierstrass degree of the
conjugacy function, defined below, we will give sufficients
conditions for the existence of an indifferent periodic point on the boundary
$S_{\sigma}(0)$ for more general $f$.


The Weierstrass degree is defined as follows. Let $K$ be an
algebraically closed complete non-Archimedean field. Let $U\subset
K$ be a rational closed disc, and let $h$ be a power series which
converges on $U$. For any disc $V\subseteq U$, the
\emph{Weierstrass degree} or simply the degree $\deg (h,V)$ of $h$
on $V$ is the number $d$ (if $V$ is closed) or $d'$ (if $V$ is
open) in Proposition \ref{proposition-discdegree}. Note that if
$0\in h(V)$, then the Weierstrass degree is the same as the notion
of degree as 'the number of pre-images of a given point, counting
multiplicity'.
More information on the properties of the Weierstrass
degree can be found in \cite{Benedetto:2003a}.

The following lemma shows that a shift of the value of Weierstrass
degree from $1$ to $d>1$,  of the conjugacy function on a sphere
$S$, reveals the existence of an indifferent periodic point on the sphere $S$.

\begin{lemma}\label{lemma weierstrass degree and per points}
Let $K$ be a complete non-Archimedean field. Let $f$ be a
linearizable power series of the form $f(x)=\lambda x +\sum_{i\geq
2} a_i x^i\in K[[x]]$, such that $|\lambda |=1$ and $a=\sup_{i\geq
2}|a_i|^{1/(i-1)}$. Let $g$ be the corresponding conjugacy
function. Let $\tau <1/a$. Suppose $\deg (g,D_{\tau}(0))=1$ and
$\deg(g,\overline{D}_{\tau}(0)) = d>1$ in the algebraic closure
$\widehat{K}$. Then, $f$ has an indifferent periodic point in $\widehat{K}$ on
the sphere $S_{\tau}(0)$ of period $\kappa\leq d $. In particular,
$D_{\tau}(0)$ is the linearization disc of $f$ about the fixed point at
the origin.
\end{lemma}

\begin{proof}
Let  $\tau <1/a$, and suppose $\deg (g,D_{\tau}(0))=1$ and
$\deg(g,\overline{D}_{\tau}(0)) = d>1$. Note that by definition
$\tau\in |\widehat {K}^*|$. Hence, $S_{\tau}(0)$ is rational and
non-empty in the algebraic closure $\widehat{K}$. The proof that
there is a periodic point in $\widehat{K}$  on the sphere
$S_{\tau}(0)$, goes as follows. Since the conjugacy $g$ maps the
closed disc $\overline{D}_{\tau}(0)$ onto itself exactly
$d$-to-$1$ (counting multiplicity), and the open disc
$D_{\tau}(0)$ one-to-one onto itself, there exist at least one
point $\hat{x}\in S_{\tau}(0)$ such that $g(\hat{x})=0$. In view
of the Schr\"{o}der functional equation
\begin{equation}\label{equation causing periodic point}
|g(f(\hat{x}))-g(f^{\circ n}(\hat{x}))|=|g(\hat{x})||\lambda
-\lambda ^n|=0,
\end{equation}
for all $n\geq 1$. Recall that $ 1/a>\tau$. By Lemma \ref{lemma f
one-to-one} $f:\overline{D}_{\tau}(0)\rightarrow
\overline{D}_{\tau}(0)$ is bijective in $\widehat{K}$.
The same is true for all the iterates $f^{\circ n}$, $n\geq 1$.
Moreover $f^{\circ n}(0)=0$ and consequently $f^{\circ n}$ can
have no zeros on the sphere $S_{\tau}(0)$ for any $n\geq 1$. In
particular $f^{\circ n}(\hat{x})\neq 0$ for all $n\geq 1$.

Let $y=g(f(\hat{x}))$. Then $y\in \overline{D}_{\tau}(0)$. The
equation $g(x)=y$ can have only $d$ solutions on
$\overline{D}_{\tau}(0)$ and we conclude from (\ref{equation
causing periodic point}) that we must have that $f^{\circ (\kappa
+1)}(\hat{x})=f(\hat{x})$ for some $\kappa \leq d$. This shows the
existence of a periodic point $\hat{x}\in\widehat{K}$ of
period $\kappa\leq d $.

Finally, since $f^{\circ \kappa}$ is one-to-one on  $\overline{D}_{\tau}(0)$,
it follows by the first statement of proposition \ref{proposition-discdegree} that 
$| (f^{\circ \kappa})'(\hat{x})|=1$. This proves that $\hat{x}$ is indifferent.

\end{proof}

Let us return to the case $f\in
\mathcal{F}_{\lambda,a}^p(K)$.
By Theorem \ref{theorem convergence divisible by p}, the
Weierstrass degree of $g$ on the open disc $\deg
(g,D_{\sigma}(0))=1$.
In the following lemma we find a necessary and sufficient
condition that the Weierstrass degree on the closed disc
$\deg(g,\overline{D}_{\sigma}(0))>1$. Again, the integer $k'$,
defined by (\ref{definition k'}), plays a significant role.
\begin{lemma}\label{lemma sufficient cond per point}
Let $f\in \mathcal{F}_{\lambda,a}^p(K)$. Then, in
$\widehat{K}$, $\deg(g,\overline{D}_{\sigma}(0))>1$ if
and only if the coefficient, $b_{k'}$, of $g$ satisfies
\begin{equation}\label{condition degree >1}
|b_{k'}|=a^{k'-1}/|1-\lambda ^m|.
\end{equation}
Moreover, if (\ref{condition degree >1}) holds, then
$\deg(g,\overline{D}_{\sigma}(0))=k'$.
\end{lemma}

\begin{proof}
By the estimate (\ref{estimate b_k (k,p)>1}) we always have
\[
|b_{k'}|\leq a^{k'-1}/|1-\lambda ^m|.
\]
If $|b_{k'}|=a^{k'-1}/|1-\lambda ^m|$. Then, we have equality in
(\ref{estimate sigma (k,p)>1}) if and only if $k=1$ or $k=k'$.
Consequently, in the algebraic closure $\widehat{K}$, $g$
maps the closed disc $\overline{D}_{\sigma}(0)$ onto
$\overline{D}_{\sigma}(0)$ exactly $k'$-to-$1$ counting
multiplicity. It follows that the Weierstrass degree
$\deg(g,\overline{D}_{\sigma}(0))=k'>1$.

On the other hand, if $|b_{k'}|<a^{k'-1}/|1-\lambda ^m|$, then we
have equality in (\ref{estimate sigma (k,p)>1}) if and only if
$k=1$. Consequently, $\deg(g,\overline{D}_{\sigma}(0))=1$ in this
case.
\end{proof}


\vspace{1.5ex}\noindent If $f$ is of the form as in Example
\ref{example a k'x^k'}, then Lemma \ref{lemma sufficient cond per
point} applies and we have.

\begin{theorem}
Let $f(x)=\lambda x +a_{k'}x^{k'}$, where $a_{k'}\neq 0$. Then,
the linearization disc of $f$ about the origin is equal to
$D_{\sigma}(0)$. In in the algebraic closure
$\widehat{K}$, we have
$\deg(g,\overline{D}_{\sigma}(0))=k'$. Moreover, $f$ has an
indifferent periodic point in $\widehat{K}$ on the sphere
$S_{\sigma}(0)$ of period $\kappa\leq k'$, with multiplier
$\lambda ^{\kappa }$.
\end{theorem}
\begin{proof}
It remains to prove that the multiplier of the periodic point is
of the form $\lambda ^{\kappa }$.  Because the degree of the
nonlinear monomials of $f\in
\mathcal{F}_{\lambda,a}^p(K)$ are all divisible by char
$K=p$, the derivative $\left (f^{\circ n} \right )'(x)
=\lambda ^n$ for all $x\in \widehat{K}$ and all $n\geq
1$. It follows that $\hat{x}$ is an indifferent periodic point of
period $\kappa \leq d$, with multiplier $\lambda ^{\kappa}$.
\end{proof}

In fact, this result can be generalized according to the following
theorem.

\begin{theorem}
Let $f\in \mathcal{F}_{\lambda,a}^p(K)$. Suppose
$a=|a_{k'}|^{1/(k'-1)}$ and $|a_i|<a^{i-1}$ for all $i<k'$. Then,
$D_{\sigma}(0)$ is the linearization disc of $f$ about the origin. In
$\widehat{K}$ we have  $\deg
(g,{\overline{D}_{\sigma}(0)})=k'$.
Furthermore, $f$ has an indifferent periodic point in
$\widehat{K}$ on the sphere $S_{\sigma}(0)$ of period
$\kappa\leq k'$, with multiplier $\lambda ^{\kappa }$.
\end{theorem}
\begin{proof} Let $f\in \mathcal{F}_{\lambda,a}^p(K)$. Suppose $|a_{k'}|=a^{k'-1}$ and $|a_i|<a^{i-1}$
for all $i<k'$. For $k=k'$ and $l=1$ the equation
(\ref{index-equations}) has the solution $\alpha_{k'}=1$,
$\alpha_j=0$ for $j<k'$. Hence, for $k=k'$, the recursion formula
(\ref{bk-equation}) contains the term
\begin{equation}\label{dominating term k=k'}
b_1a_{k'}/(1-\lambda ^{k'-1}),
\end{equation}
where $b_1=1$, $|a_k'|=a^{k'-1}$, and $|1-\lambda
^{k'-1}|=|1-\lambda ^m|$. The minimality of $k'$ and the
assumption that $|a_i|<a^{i-1}$ for all $i<k'$ yields in view of
Lemma \ref{lemma b_k estimate p|k} that the term (\ref{dominating
term k=k'}) is strictly greater than all the other terms in the
recursion formula (\ref{bk-equation}). Hence, by ultrametricity,
$|b_k'|=a^{k'-1}|1-\lambda ^m|^{-1}$ as required. \end{proof}

\begin{corollary}
Let $f\in \mathcal{F}_{\lambda,a}^p(K)$  be of the form
\[
f(x)=\lambda x +\sum_{i\geq k'}a_ix^i, \quad
a=|a_{k'}|^{1/(k'-1)}>0.
\]
Then, $D_{\sigma}(0)$ is the linearization disc of $f$ about the origin.
In $\widehat{K}$ we have  $\deg
(g,{\overline{D}_{\sigma}(0)})=k'$.
Furthermore, $f$ has an indifferent periodic point in
$\widehat{K}$ on the sphere $S_{\sigma}(0)$ of period
$\kappa\leq k'$, with multiplier $\lambda ^{\kappa }$.
\end{corollary}





\section*{Acknowledgements}

I would like to thank
 Prof. Andrei Yu. Khrennikov for fruitful discussions and for introducing me to the
theory of $p$-adic dynamical systems and the problem on linearization. I thank
Robert L. Benedetto for his hospitality
during my visit in Amherst,  for fruitful
discussions,
and consultancy on non-Archimedean analysis.




\begin{thebibliography}{10}

\bibitem{Arrowsmith/Vivaldi:1993}
D.~K. Arrowsmith and F.~Vivaldi.
\newblock Some $p$-adic representations of the {S}male horseshoe.
\newblock {\em Phys. Lett. A}, 176:292--294, 1993.

\bibitem{Arrowsmith/Vivaldi:1994}
D.~K. Arrowsmith and F.~Vivaldi.
\newblock Geometry of $p$-adic {S}iegel discs.
\newblock {\em Physica D}, 71:222--236, 1994.

\bibitem{Beardon:1991}
A.~F. Beardon.
\newblock {\em Iteration of Rational Functions}.
\newblock Springer-Verlag, Berlin Heidelberg New York, 1991.

\bibitem{Benedetto:2001b}
R.~Benedetto.
\newblock Reduction dynamics and {J}ulia sets of rational functions.
\newblock {\em J. Number Theory}, 86:175--195, 2001.

\bibitem{Benedetto:2003a}
R.~Benedetto.
\newblock {N}on-{A}rchimedean holomorphic maps and the {A}hlfors {I}slands
  theorem.
\newblock {\em Amer. J. Math.}, 125(3):581--622, 2003.

\bibitem{Bezivin:2004b}
J-P. B\'{e}zivin.
\newblock Fractions rationnelles hyperboliques $p$-adiques.
\newblock {\em Acta Arith.}, 112(2):151--175, 2004.

\bibitem{BoschGuntzerRemmert:1984}
S.~Bosch, U.~G\"untzer, and R.~Remmert.
\newblock {\em Non-Archimedean analysis: {A} systematic approach to rigid
  analytic geometry}.
\newblock Springer-Verlag, Berlin, 1984.

\bibitem{Brjuno:1971}
A.~D. Brjuno.
\newblock Analytical form of differential equations.
\newblock {\em Trans. Moscow Math. Soc.}, 25,26:131--288,199--239, 1971,1972.

\bibitem{Carleson/Gamelin:1991}
L.~Carleson and T.~Gamelin.
\newblock {\em Complex Dynamics}.
\newblock Springer-Verlag, Berlin Heidelberg New York, 1991.

\bibitem{Cassels:1986}
J.W.S. Cassels.
\newblock {\em Local Fields}.
\newblock Cambridge University Press, Cambridge, 1986.

\bibitem{FresnelvanderPut:1981}
J.~Fresnel and M.~van~der Put.
\newblock {\em G\'{e}om\'{e}trie analytique rigide et applications}.
\newblock Birkh\"{a}user, Boston, 1981.

\bibitem{Herman:1986}
M.~Herman.
\newblock Recent results and open questions on {S}iegel's linearization theorem
  of complex analytic diffeomorphisms of $\mathbb{C}^n$ near a fixed point.
\newblock In {\em Proc. VIII-th International Congress on Mathematical Physics
  1986}, pages 138--184. World Scientific, 1987.

\bibitem{HermanYoccoz:1981}
M.~Herman and J.-C. Yoccoz.
\newblock Generalizations of some theorems of small divisors to non archimedean
  fields.
\newblock In {\em Geometric Dynamics}, volume 1007 of {\em LNM}, pages
  408--447. Springer-Verlag, 1981.

\bibitem{Hsia:2000}
L.~Hsia.
\newblock Closure of periodic points over a non-archimedean field.
\newblock {\em J. London Math. Soc.}, 62(2):685--700, 2000.

\bibitem{Khrennikov:2001a}
A.~Yu. Khrennikov.
\newblock Small denominators in complex $p$-adic dynamics.
\newblock {\em Indag. Mathem.}, 12(2):177--188, 2001.

\bibitem{KhrennikovNilsson:2001}
A.~Yu. Khrennikov and M.~Nilsson.
\newblock On the number of cycles for $p$-adic dynamical systems.
\newblock {\em J. Number Theory}, 90:255--264, 2001.

\bibitem{Koblitz:1984}
N.~Koblitz.
\newblock {\em $p$-adic numbers, $p$-adic analysis, and zeta-functions}.
\newblock Springer-Verlag, New York, second edition, 1984.

\bibitem{Li:1996a}
H-C. Li.
\newblock $p$-adic dynamical systems and formal groups.
\newblock {\em Compos. Math.}, 104:41--54, 1996.

\bibitem{Li:2002a}
H-C. Li.
\newblock On heights of $p$-adic dynamical systems.
\newblock {\em Proc. Amer. Math. Soc.}, 130(2):379--386, 2002.

\bibitem{Lindahl:2004cpp}
K.-O. Lindahl.
\newblock Estimates of linearization discs in $p$-adic dynamics with
  application to ergodicity.
\newblock http://arxiv.org/abs/0910.3312.

\bibitem{Lindahl:2004}
K.-O. Lindahl.
\newblock On {S}iegel's linearization theorem for fields of prime
  characteristic.
\newblock {\em Nonlinearity}, 17(3):745--763, 2004.

\bibitem{Lindahl:2007}
K.-O. Lindahl.
\newblock {\em On the linearization of non-Archimedean holomorphic functions
  near an indifferent fixed point}.
\newblock PhD thesis, V\"{a}xj\"{o} University, 2007.
\newblock Introduction and summary available at
  \texttt{http://www.diva-portal.org/}.

\bibitem{Lubin:1994}
J.~Lubin.
\newblock Non-archimedean dynamical systems.
\newblock {\em Compos. Math.}, 94:321--346, 1994.

\bibitem{Milnor:2000}
J.~Milnor.
\newblock {\em Dynamics in {O}ne {C}omplex {V}ariable}.
\newblock Vieweg, Braunschweig, 2nd edition, 2000.

\bibitem{Rivera-Letelier:2003thesis}
J.~Rivera-Letelier.
\newblock Dynamique des functionsrationelles sur des corps locaux.
\newblock {\em Ast\'{e}risque}, 287:147--230, 2003.

\bibitem{Rivera-Letelier:2003}
J.~Rivera-Letelier.
\newblock Espace hyperbolique $p$-adique et dynamique des fonctions
  rationnelles.
\newblock {\em Compos. Math.}, 138(2):199--231, 2003.

\bibitem{Siegel:1942}
C.~L. Siegel.
\newblock Iteration of analytic functions.
\newblock {\em Ann. of Math.}, 43:607--612, 1942.

\bibitem{DeSmedtKhrennikov:1997}
S.~De Smedt and A.~Khrennikov.
\newblock Dynamical systems and theory of numbers.
\newblock {\em Comment. Math. Univ. St. Pauli}, 46(2):117--132, 1997.

\bibitem{Yoccoz:1988}
J.-C. Yoccoz.
\newblock Lin\'{e}arisation des germes de difféomorphismes holomorphes de
  $(\mathbb{C}, 0)$.
\newblock {\em C. R. Acad. Sci. Paris S\'{e}r. I Math.}, 306(1):55--58, 1988.

\end{thebibliography}

\end{document}